\documentclass[12pt]{article}

\usepackage{latexsym}
\usepackage{booktabs}
\usepackage{caption}
\usepackage{ltablex}
\usepackage{hyperref}
\usepackage{amssymb}
\usepackage{amsmath}
\usepackage{amsfonts}
\usepackage{mathtools}
\usepackage{array,multirow,makecell}

\usepackage[font={small},width=1\linewidth]{caption}

\usepackage{comment}
\usepackage{tikz}
\usetikzlibrary{arrows.meta}

\usepackage{pgfplots}
\pgfplotsset{compat=newest}
\usepgfplotslibrary{fillbetween}

\usepackage{amsthm,cite}
\usepackage{dsfont}
\usepackage{enumitem}

\usepackage{tipa}
\usepackage{mathrsfs,bbm,subfigure}
\usepackage{url}

\newtheorem{nummer}{ }%[section]
\newtheorem{thm}[nummer]{\bf Theorem}
\newtheorem{prop}[nummer]{\bf Proposition}
\newtheorem{lem}[nummer]{\bf Lemma}
\newtheorem{cor}[nummer]{\bf Corollary}

\newtheorem{exa}{\bf Example}

\theoremstyle{definition}
\newtheorem{rmk}{\bf Remark}

\newcommand{\dis} {\displaystyle}

\newcommand{\Q}{\mathds{Q}}

\newcommand{\Z}{\mathds{Z}}

\newcommand{\CN}{\mathcal C_N}

\newcommand{\dc}[1]{\textcolor{red}{#1}}
\newcommand{\cb}[1]{\textcolor{blue}{#1}}

\makeatletter
\def\opargproof[#1]{\par\noindent {\bf #1 }}
 \makeatother

\definecolor{darkgreen}{rgb}{0,.6,0}

\begin{document}
\begin{center}
{\LARGE\bf Integer triangles with a rational ratio of circumcircle radius to excircle radius}

\medskip

{\small Lorenz Halbeisen}\\[1.2ex] 
{\scriptsize Department of Mathematics, ETH Zentrum,
R\"amistrasse\;101, 8092 Z\"urich, Switzerland\\ lorenz.halbeisen@math.ethz.ch}\\[1.8ex]
{\small Norbert Hungerb\"uhler}\\[1.2ex] 
{\scriptsize Department of Mathematics, ETH Zentrum,
R\"amistrasse\;101, 8092 Z\"urich, Switzerland\\ norbert.hungerbuehler@math.ethz.ch}\\[1.8ex]
{\small Arman Shamsi Zargar}\\[1.2ex] 
{\scriptsize Department of Mathematics and Applications, Faculty of Mathematical Sciences, 
	University of Mohaghegh Ardabili,
Ardabil, Iran\\ zargar@uma.ac.ir}
\end{center}

\hspace{5ex}{\small{\it key-words\/}: circumcircle, excircle, elliptic curve}

\hspace{5ex}{\small{\it 2020 Mathematics Subject Classification\/}: {\bf 11D72} 11G05}%% 
% 11D72 Diophantine equations in many variables
% 11G05 Elliptic curves over global fields 

\begin{abstract}\noindent
{\small We consider the problem of finding integer triangles with $R/r$ a positive rational, where $R$ and $r$ are the radii of the circumcircle and an excircle, respectively.
We show that for general triangles $R/r>1/4$ applies.
The equation $R/r=N$ turns out to be related to the elliptic curve $\mathcal{E}_N$ given by $v^2=u^3+2(2N^2+2N-1)u^2-(4N-1)u$.
If $N>1/4$ is rational, then the torsion group of $\mathcal{E}_N$ is $\Z/2\Z\times\Z/6\Z$ if $N(N+2)$ is a square and $\Z/6\Z$ otherwise.
We show that a rational triangle with rational ratio $R/r=N$ exists if and only if $N>1/4$ and there exists a rational non-torsion point on the curve $\mathcal{E}_N$ which satisfies a certain 
condition. Furthermore, we show that the rank of $\mathcal{E}_N$ is positive when 
$N = m^2 \pm 1>1/4$ for a rational $m$. We also show that on 
every curve $\mathcal{E}_N$ whose rank is positive, there are  infinitely many 
rational points which lead to infinitely many non-similar integer triangles with $R/r=N$.}
\end{abstract}

\section{Introduction}\label{sec-1}
Consider a triangle with sides of integer length $f, g, h$. 
Prominent circles associated with the triangle are the circumcircle
which passes through the three vertices, the incircle, and the three excircles which have the three sides as tangents.
% (and differ from the incircle).  %%% Everybody knows what in- and excircles are!
The radii of the circumcircle $C$ and one of the three excircles $E$ are denoted in the following by $R$ and $r$, respectively (see Figure~\ref{fig1}).
Let $d$ be the distance between the 
centres of these two circles. 
It is a standard result 
in triangle geometry that 
\begin{equation}\label{euler}
d^2 = R(R+2r)
\end{equation}
(see, e.g.,~\cite[Theorem 295]{johnson}).
% Proof see also https://www.cut-the-knot.org/triangle/EulerIO.shtml
%
%This relation should not be confounded with the Euler-Chapple formula for the circumcircle and the incircle. \dc{Give a citation to \cite{Chapple}, in its suitable place?}
%Euler's theorem in geometry (first published by 
%Chapple~\cite{Chapple}), it is a standard result 
%in triangle geometry that 
\begin{comment}
As a matter of fact we would like to mention that this equation 
also holds in the case when $r$ is the radius of the incircle.  
%% No, it does not. The formula for incircle is $d^2 = R(R-2r)$, with a minus.
\end{comment}
\begin{comment}
\dc{What about using the notations $C_R$ and $C_r$ instead of the circles $C$ and $E$ respectively throughout the paper? Also, is it better to use different color for the circles and lines as the next shapes?}
\cb{I think the notation $C_R$ and $C_r$ is confusing, since $r$ and $R$
stand for real numbers (compare, for example, with the expression $R/r$).}
\end{comment}

\begin{figure}[htbp]
\begin{center}
\begin{tikzpicture}[line cap=round,line join=round,x=22,y=22]
\clip(-2.54,-2.06) rectangle (6.84,4.2);
\draw[line width=1.pt,fill=black,fill opacity=0.1] (-1.88,-0.78) -- (2.8,-0.86) -- (3.24,2.08) -- cycle;
\draw [line width=.5pt,domain=-1.88:6.839999999999995] plot(\x,{(-3.8008-0.08*\x)/4.68});
\draw [line width=.5pt,domain=-1.88:6.839999999999995] plot(\x,{(--1.3832--2.86*\x)/5.12});
\draw [line width=0.8pt] (4.909180215269176,0.9258364631531464) circle (1.8216247009408981);
\draw [line width=0.8pt] (0.49091457403004113,0.9885025807574094) circle (2.957843352088421);
%\draw [line width=2.pt] (-1.88,-0.78)-- (2.8,-0.86);
%\draw [line width=2.pt] (2.8,-0.86)-- (3.24,2.08);
%\draw [line width=2.pt] (3.24,2.08)-- (-1.88,-0.78);
\begin{small}
\draw [line width=.5pt] (4.909180215269176,0.9258364631531464)-- (4.878045879970928,-0.8955221517943747) node[midway,right,xshift=-2pt,yshift=2pt]{$r$};
\draw [line width=.5pt] (0.49091457403004113,0.9885025807574094)-- (-1.88,-0.78)node[midway,above,xshift=-0pt,yshift=0pt]{$R$};
\draw [fill=white] (-1.88,-0.78) circle (1.5pt);
\draw [fill=white] (2.8,-0.86) circle (1.5pt);
\draw [fill=white] (3.24,2.08) circle (1.5pt);
\draw [fill=white] (4.909180215269176,0.9258364631531464) circle (1.5pt);
\draw [fill=white] (4.878045879970928,-0.8955221517943747) circle (1.5pt);
\draw[color=black] (6.52,2.33) node {$E$};
\draw [fill=white] (4.0208312593009365,2.5161674612501326) circle (1.5pt);
\draw [fill=white] (3.10761954299925,1.1954578554949868) circle (1.5pt);
\draw[color=black] (-1.72,3.29) node {$C$};
\draw [fill=white] (0.49091457403004113,0.9885025807574094) circle (1.5pt);
\draw[color=black] (0.5,-1.1) node {$g$};
\draw[color=black] (2.72,0.37) node {$h$};
\draw[color=black] (0.78,0.44) node {$f$};
\end{small}
\end{tikzpicture}
\caption{A triangle with sides $f,g,h$, circumcircle $C$ with radius $R$, excircle $E$ touching $h$ from outside with radius $r$.}\label{fig1}
\end{center}
\end{figure}
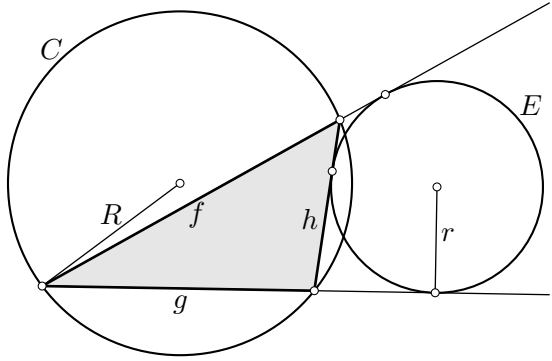
As a consequence of equation~\eqref{euler}, we obtain the following.

\begin{lem}
	\label{lem1}
	$R/r>1/4$.
\end{lem}
\begin{proof}
If $r$ is the radius of an excircle and $r > R$, then, since the vertices of the triangle lie outside the excircle, we have $d + R > r$, i.e., 
$d > r - R$. By \eqref{euler}, we obtain
$$d^2 = R^2 + 2Rr > R^2 - 2Rr + r^2,$$ 
which implies $4R > r$, i.e., $R/r > 1/4$. 
\end{proof}
\begin{comment}
\dc{What about assuming the condition $R>r$ in the above proof (to be consistent to Fig~1 and Fig~4 and the example in Section~7) and prove like this: $\ldots$}
\cb{If $R>r$, then $R/r>1$ and $1>1/4$.}
\end{comment}

Equally well known are the formulas 
\begin{equation}
	\label{R,r}
	R=\frac{fgh}{4\Delta}, \quad r=\sqrt{\frac{s(s-f)(s-g)}{s-h}}
\end{equation}
where $r$ is the radius of the excircle touching $h$ from outside, $s=(f+g+h)/2$ is the semi-perimeter, and $\Delta=\sqrt{s(s-f)(s-g)(s-h)}$ the area of the triangle. Hence, we have
\begin{equation}
	\label{R/r}
	\frac{R}{r}=\frac{2fgh}{(f+g+h)(f-g-h)(g-h-f)}.
\end{equation}
Observe that for non-degenerate triangles, the triangle inequality applies, i.e., the denominator on the right hand side of~\eqref{R/r} is different from $0$. 
\begin{comment}
\dc{What about writing ``strictly positive" instead of ``different from 0"?}
\cb{We just need different from $0$, otherwise, the reader may think that
``strictly positive" is essential.}
\end{comment}

It should be noted that we  can always scale a triangle with rational sides to a similar triangle with integer sides without changing the value of $R/r$.
It is interesting to speculate whether the ratio $R/r$ could ever be an integer $N$ for a  triangle with integer sides. In \cite{MacLeod}, MacLeod studied a similar problem for such  triangles with $R/\rho$ an integer, where $R$ and $\rho$ are the radii of the circumcircle and incircle, respectively. He showed that these  triangles are relatively rare, and by brute force found only a handful such integer triangles. It is interesting to see that our problem, although very similar to MacLeod's, 
behaves quite differently. Indeed, we will see in Section~\ref{sec-2} that our problem leads to finding certain rational non-torsion points on the elliptic curve $\mathcal{E}_N$ given by $v^2=u^3+2(2N^2+2N-1)u^2-(4N-1)u$. In other words, the existence of integer triangles with rational ratio $R/r=N$ is equivalent to the existence of special non-torsion points on the curve $\mathcal{E}_N$. The torsion group of $\mathcal{E}_N$ and the torsion points which are needed later are determined in Section~\ref{sec-torsion}.
Section~\ref{sec-4} identifies the non-torsion points that correspond to real triangles.
So, a necessary condition for the existence of such triangles is that the rank of $\mathcal{E}_N$ is positive. We will show in Section~\ref{sec-3} that this is in particular the case for the values  $N=m^2\pm 1$ for any rational number $m\geq 2$.
We then show in Section~\ref{sec:triangles} that whenever the rank of $\mathcal{E}_N$ is positive for a rational $N>1/4$, 
  infinitely many pairwise non-similar integer triangles with $R/r=N$ exist. 
We conclude in Section~\ref{sec:poncelet} with an application of our results to Poncelet's porism.
%In the next section, we first show how we can associate a rational point on $\mathcal{E}_N$ (for $N>1/4$) to a rational triangle with given value $R/r=N$. 
\begin{comment}
\dc{What about saying a bit about the sections~3--7?}
\cb{Done.}
\end{comment}

\section{The construction of the elliptic curve $\mathcal{E}_N$}\label{sec-2}

%We must find integers $f, g, h$ such that
Suppose we have a triangle with rational sides $f,g,h$ such that the ratio of the radius $R$ of the circumcircle, and the radius $r$ of the excircle touching $h$ from outside
\begin{equation}
	\label{4}
	N=\frac Rr=\frac{2fgh}{(f+g+h)(f-g-h)(g-h-f)}
\end{equation}
is a rational number.
%In all cases we look to express $N$ as a ratio of two homogeneous
%cubics in \cb{the $3$ variables $f,g$ and $h$.} \textcolor{red}{???}
If we multiply~\eqref{4} by the denominator and move all terms to the left, we can rewrite~\eqref{4} as a cubic equation in the variables $f,g,h$:
\begin{equation*}
	%\label{5}
	Nf^3-N(g-h)f^2-(g^2N+2gh(N-1)+h^2N)f+N(g-h)(g+h)^2=0.
\end{equation*}
%We cannot get very far with this form, but we can proceed quickly 
Now, we replace $h$ by $2s - f - g$ and obtain the following quadratic equation in $f$:
\begin{equation}
	\label{6}
	gf^2+\big(g^2+2s(2N-1)g-4s^2N\big)f-4(g-s)s^2N=0.
\end{equation}
Since $f$ is rational, the quadratic equation~\eqref{6} must have a discriminant
which is a rational square, so that there must be a rational $d$ with
\begin{equation*}
	%\label{7}
	d^2=g^4+4s(2N-1)g^3+4s^2(4N^2-2N+1)g^2-32N^2s^3g+16N^2s^4.
\end{equation*}
If we substitute $d=s^2y$ and $g=sx$, we get the quartic equation
\begin{equation}
	\label{8}
\CN:	y^2=x^4+4(2N-1)x^3+4(4N^2-2N+1)x^2-32N^2x+16N^2.
\end{equation}
Now we apply the birational transformation 
\begin{align}
x&= -\frac{4 N (2 N u+v)}{(u-1) (4 N+u-1)},  \label{eq-x}\\
y&= -\frac{4 N \left(4 N+u^2-1\right) \left(8 N^2 u+4 N u+4 N v-4 N+u^2-2 u+1\right)}{(u-1)^2 (4 N+u-1)^2}\label{eq-y}
\end{align}
to $\CN$. This transformation leads to % resulting equation %factors in a quartic (\dc{I don't understand ``factors in a quartic" here}) 
a cubic  curve in $u$ and $v$ which is given by 
\begin{equation}\label{9}
	\mathcal{E}_N: v^2=u^3+2(2N^2+2N-1)u^2-(4N-1)u.
\end{equation}
The curve $\mathcal{E}_N$ is regular for positive rational $N\neq1/4$. Conversely, we get from $\mathcal{E}_N$ to $\CN$ with the birational transformation
%% I don't quite understand this:  in what sense is \eqref{10} an elliptic curve?
% with the transformation
% \begin{equation}
%	\label{10}
%	\frac{g}{s}=-\frac{4N(v+2Nu)}{(u-1)(u+4N-1)}.
%\end{equation}
%The curve \eqref{10} remains elliptic for any positive integer $N$.
%
%Indeed, the rational map $\mathcal C\ni (x,y)\mapsto (u,v)\in \mathcal{E}_N$
%is given by
\begin{align}
	u&=-\frac{8N^2x+2Nx^2-8N^2+2Ny-x^2}{x^2}, \\
	v&=-\frac{2N(x-2)(8N^2x+2Nx^2-8N^2+2Ny-x^2)}{x^3}.
\end{align}
%$$
%and the inverse map by
%$$
%\begin{aligned}
%x&=-\frac{4N(v+2Nu)}{(u-1)(u+4N-1)}, \\
%y&=-\frac{4N(u^2+4N-1)(8N^2u+4Nu+4Nv+u^2-4N-2u+1)}{(u-1)^2(u+4N-1)^2}.	
%\end{aligned}
Thus, starting with the rational triangle with sides $f,g,h$ and rational ratio $R/r=N$, we can find a rational point $(u,v)$ on
the cubic curve $\mathcal{E}_N$. Using the equations in the opposite direction, we can also reconstruct the sides $f,g,h$ from $(u,v)$, at least up to scaling.
To do so, use~\eqref{eq-x} and~\eqref{eq-y} to find $(x,y)$ on $\CN$ and then the formulas in Theorem~\ref{th:1} to compute the sides of the triangle.
However, not every rational point $(u,v)$ on $\mathcal{E}_N$ leads to numbers $f,g,h$ which are the sides of a triangle.
So our next task is to identify those rational points $(u,v)$ on $\mathcal{E}_N$ that lead to {\em positive\/} values of $f,g,h$ which
satisfy the {\em triangle inequalities} $f+g>h, g+h>f$ and $h+f>g$. The three triangle inequalities can be equivalently expressed
by the single inequality
$$
 (f + g - h) (f - g + h)(-f + g + h) (f + g + h)>0
$$
or, also equivalently,
$$
(f^2+g^2+h^2)^2>2(f^4+g^4+h^4).
$$
Observe that if we replace $N$ by $-N$ in \eqref{9}, we get 
MacLeod's curve~\cite[eq.~(9)]{MacLeod}. Thus, our discussion 
also covers MacLeod's problem in the case when $N$ 
is a rational number.
%In some sense, our problem covers the related MacLeod's Problem for negative integers $N$. \textcolor{red}{... we would have to explain this.} \cb{In other words, by changing $N$'s in MacLeod's family $v^2=u^3+2(2N^2-2N-1)u^2+(4N+1)u$ to $-N$'s we get our family being studied in this work.}  

Also notice that if $R/r$ is rational for a rational triangle for just one of its three excircles, then all ratios of radii of the three excircles, the incircle, and the circumcircle are rational.
Another remark is that the formula~\eqref{4} is symmetric in $f$ and $g$. Hence, if we start with a rational point $(u,v)$ on $\mathcal{E}_N$ which
corresponds to a triangle with sides $f,g,h$, then we can find another rational point $(u',v')$ on $\mathcal{E}_N$ which belongs to the mirrored triangle with sides $g,f,h$.

\section{Torsion points of $\mathcal{E}_N$}\label{sec-torsion}
One can easily check that the torsion group of $\mathcal{E}_N$ for rationals $N>1/4$ contains $\Z/6\Z$ with torsion points $$T_2=(0, 0), \ T_3^{\pm}=(1,\pm 2N), \ T_6^{\pm}=(1-4N,\pm 2N(4N - 1))$$
of orders~$2$, $3$, $6$ respectively, see \cite[Section~2]{MacLeod}.
However, $T_2$ corresponds via~\eqref{eq-x} to $x=0$ and hence $g=0$, which does not give a real triangle.
For $T_3^{\pm}$ and $T_6^{\pm}$, the expression~\eqref{eq-x} is not defined. 
%Substituting these points into \eqref{10} gives rise to $g/s=0$ or undefined, none of which lead to a practical solution of the problem. 
Thus, we must look at the second type of rational points---those of infinite order. %In effect, 
In particular, since the ranks of $\mathcal{E}_N$ are zero for the following integers $N$, no integer or rational triangle with
$R/r=N$ exists:
$$
N=
1,
2,
4,
6,
7,
9,
12,
14,
16,
18,
19,
20,
21,
22,
25,
28,
30,
\ldots
$$
Also notice that the ranks of the curves $\mathcal{E}_{1}$, $\mathcal{E}_{1/2}$ and $\mathcal{E}_{1/3}$ are zero, and recall that $R/r>1/4$. Therefore no rational triangle
exists with $r/R$ an integer.

Let us now consider the case when $N>1/4$ is a rational and $N(N+2)=M^2$ for some $M\in\Q$. In this case, $N$ cannot be an integer, but, for example, for $N = 2/3$ we get $N(N + 2) =(4/3)^2$. If
$N(N + 2) = M^2$, then $\mathcal{E}_N$ has the two additional torsion points of order~$2$
$$\left(1-2N(N+1)\pm 2 NM,0\right),$$
which shows that the torsion group of $\mathcal{E}_N$ in the case when $N(N+2)=M^2$ is $\Z/2\Z \times \Z/6\Z$.

\begin{comment}
We show that $\Z/12\Z$ is not possible for rational $N>1/4$. 
Suppose towards a contraction that  $P = (u, v)$ is a point of order~$12$.  This is equivalent to saying that the point $[2]P$ is of order~$6$, i.e.,
$$\frac{(u^2+4N-1)^2}{4\big(u^3+2(2N^2+2N-1)u^2-(4N-1)u\big)}=1-4N.$$
The condition $N(N+2)=M^2$ is accomplished by 
$$N=-2\frac{t^2}{t^2-1}$$
for any rational $t\neq \pm 1$. Solving the above equation in $u$, we get the following values for $u$:
\begin{multline*}
 \frac{1}{t^2-1}\Big(4t^2(t^2-1)\sqrt{A}+9t^2-1\pm 2\frac{t}{t^2-1}\sqrt{2(9t^2-1)\big(3t^4-4t^2+1+(t^2-1)^3\sqrt{A}\big)}\Big),
\end{multline*}
and
\begin{multline*}
\frac{-1}{t^2-1}\Big(4t^2(t^2-1)\sqrt{A}-9t^2+1\pm  2\frac{t}{t^2-1}\sqrt{2(9t^2-1)\big(3t^4-4t^2+1-(t^2-1)^3\sqrt{A}\big)}\Big),
\end{multline*}
where 
$$A=\frac{(3t-1)(3t+1)}{(t-1)^3(t+1)^3}.$$
Hence, we must have $$(t-1)(t+1)(3t-1)(3t+1)\in\Q^2,$$
which, by standard transformation, is (birationally) equivalent to the rank~$0$ elliptic curve $$V^2=U^3-\frac{208}{3}U+\frac{4480}{27}$$
with torsion group $\Z/2\Z\times\Z/4\Z$. Its torsion points $(-4/3,\pm 16)$, $(44/3,\pm 48)$, which trace back to $(\pm 1,0)$, $(\pm 1/3,0)$ on the quartic, leads to the degenerate cases $N=1/4, 0, -2$, respectively. Therefore, the $\Z/12\Z$ does not happen in this case.
\end{comment}

Finally, let us consider the case when $N > 1/4$ is  rational and $N(N+2)$ is not a square. Then the torsion group of $\mathcal{E}_N$ contains $\Z/6\Z$ and is different from $\Z/2\Z\times \Z/6\Z$. Thus, it is either $\Z/6\Z$ or $\Z/12\Z$. Suppose 
%towards a contradiction 
that $P = (u, v)$ is a point of order~$12$. This is equivalent to saying that the point $[2]P$ is of order~$6$, i.e.,
$$\frac{(u^2+4N-1)^2}{4\big(u^3+2(2N^2+2N-1)u^2-(4N-1)u\big)}=1-4N.$$
Replacing $u$ by $u + (1 - 4 N)$
this corresponds to the quartic equation 
\begin{equation}\label{eqn:no_real_roots}
u^4+8 N (1 - 6 N + 8 N^2)u^2-32 (1 - 4 N)^2 
N^2 u+64 (1 - 4 N)^2 N^3=0.
\end{equation}
Using the criterion given in the table on page~53 in~\cite{Arnon},
where 
$$\delta=1048576 (1 - 4 N)^6(2N^7 + N^8)$$
and
$$L=-1024 N^3 (-1 + 4 N)^3 (-1 + N + 16N^2 + 8N^3),$$
it is easy to check that for $N>1/4$ we have
$\delta>0$ and $L\leq 0$,
which implies that equation~\eqref{eqn:no_real_roots}
does not have real roots for $N>1/4$. Hence, the 
torsion group  $\Z/12\Z$ is excluded.

In summary, we obtain the following result:
\begin{prop} 
\begin{comment}Let $M=\sqrt{1-4N}$ and denote the torsion group of $\mathcal{E}_N$ by $\mathcal T$. Then, we have
\begin{itemize}
	\item [i)] if $N$ is a positive integer, $\mathcal T\cong\Z/6\Z$.
	\item [ii)] if $N$ is rational number, 
	$${\mathcal T}\cong\left\{\begin{array}{ll}
	\Z/12\Z & \text{if } (M+1)(M-3)\in\Q^2, \\ [5pt]
	\Z/2\Z\times \Z/6\Z & \text{if } N(N+2)\in\Q^2, \\ [5pt]
	\Z/6\Z & \text{otherwise.}
	\end{array}\right.$$
\end{itemize}
\end{comment}
Let $N>1/4$ be a rational number. Then the torsion group of $\mathcal{E}_N$ is $\Z/2\Z \times \Z/6\Z$ if $N(N+2)=M^2$ for some $M\in\Q$,
and $\Z/6\Z$ otherwise.
\end{prop}

\section{The triangle inequalities}\label{sec-4}
\begin{comment}
\dc{Can this section (and maybe the next section) be rewritten even better?} \cb{How?}
\end{comment}
To simplify the notation, we will use the following abbreviations
$$
\begin{aligned}
	A_1&:=-x^2-2(2N-1)x+4N, & A_3&:=x^2-4Nx+4N, \\ 
	A_2&:=-x^2+2(2N+1)x-4N, & A_4&:=x^2+4Nx-4N, 
\end{aligned}
$$
and 
$$B:=x^4+4(2N-1)x^3+4(4N^2-2N+1)x^2-32N^2x+16N^2.$$
The following theorem now provides information on the rational points on $\CN$  which correspond to real triangles.
\begin{thm}
	\label{th:1}
	A rational triangle with rational quotient $R/r=N>1/4$ exists  if and only if $\CN$ has a rational point $(x,y)$ such that $0<x<1$. In this case, the sides of the triangle are given by
	$$f=\frac{s}{2x}(A_1-\sqrt{B}), \ g=sx, \ h=\frac{s}{2x}(A_2+\sqrt{B}),$$
	for $0<s\in\Q$.
\end{thm}
\begin{proof}
Let $f,g,h$ be sides of a rational triangle with $R/r=N$. Then, as we have seen in Section~\ref{sec-2}, $(x,y)$ must be a rational point on $\CN$.
From $g=sx$ in~\eqref{6}, we infer that $x>0$. On the other hand, we have from the triangle inequality
$$
0<f-g+h\ \ \iff\ \  2g<f+g+h\ \ \iff\ \ x=\frac gs=\frac{2g}{f+g+h}<1.
$$
For the opposite direction, assume we have a rational point $(x,y)$ on $\CN$ with $0<x<1$. We may also assume that $y>0$.
For $g=sx$, equation~\eqref{6} has the following roots for $f$:
$$f_{\pm}=\frac{s}{2x} (A_1\pm \sqrt{B}).$$
Note that by \eqref{8}, $0<\sqrt{B}=y\in\Q$.  

By $h_\mp=2s-f_\pm-g$, we get the following two values:
$$h_{\mp}=\frac{s}{2x}(A_2\mp \sqrt{B}),$$
and hence
\begin{align}
	f_\pm+g-h_\mp&=\frac{s}{x}(A_3\pm \sqrt{B}), \nonumber \\
	f_\pm+h_\mp-g&=-2s(x-1), \label{x<1}\\
	g+h_\mp-f_\pm&=\frac{s}{x}(A_4\mp \sqrt{B}). \nonumber
\end{align}
%By the second formula of \eqref{x<1} and the triangle inequality, one must have $x<1$. On the other hand $g=sx$ implies $x>0$, since $g>0$. Hence we always must have $0<x<1$. 
%Note that we look for real-life triangles.

We observe that
$$B-A_2^2=16Nx(x-1)^2>0 \Longrightarrow B>A_2^2 \Longrightarrow \sqrt{B}>A_2 \Longrightarrow h_{-}<0.$$
Hence, the root $f=f_{+}$ does not lead to a triangle. So, we will take $f=f_{-}$ and $h=h_+$. In this case, we have the following. 
\begin{comment}
\dc{Can the following items and text be
	rephrased better in shape and writing? If needed/possible, we can also shorten some
	places?}
\cb{How?}
\end{comment}
\begin{itemize}
	\item Since $N>0$ and $0<x<1$, we have $A_1>0$. Then
	$$
	\begin{aligned}
		f>0 & \Longleftrightarrow A_1-\sqrt{B}>0 \Longleftrightarrow A_1>\sqrt{B}\\
		& \Longleftrightarrow A_1^2>B  \Longleftrightarrow A_1^2-B=-16Nx(x-1)>0,
	\end{aligned}
	$$
	which is the case. 
	\item %$h=h_{+}>0$ since we have:
	%$$
	%\begin{aligned}
	%	B>A_2^2 & \Longrightarrow (\sqrt{B}-A_2)(\sqrt{B}+A_2)>0\\
	%	& \Longrightarrow \sqrt{B}-A_2>0, \sqrt{B}+A_2>0 \ \text{or} \  \sqrt{B}-A_2<0, \sqrt{B}+A_2<0.
	%\end{aligned}
	%$$
	%But the second case cannot happen, otherwise $0<\sqrt{B}<A_2$ which gives $B-A_2^2<0$, a contradiction. Therefore, the first case happens and hence $h_{+}>0$.  
	%%% Simpler:
	Since $B>A_2^2$, it follows that $h=h_+>0$.
	\item Trivially, we have $g=sx>0$.
	\end{itemize}
	 It remains to check that $f,g,h$ satisfy the triangle inequalities.
	\begin{itemize}
	\item Since $N>0$ and $0<x<1$, we have $A_3>0$. Then we have
	$$
	\begin{aligned}
		f+g-h>0 & \Longleftrightarrow A_3-\sqrt{B}>0 \Longleftrightarrow A_3>\sqrt{B}\\
		& \Longleftrightarrow A_3^2>B  \Longleftrightarrow A_3^2-B=-4x^2(x-1)(4N-1)>0,
	\end{aligned}
	$$
	which is the case. 
	\item To show $g+h-f>0$ it is enough to show that $A_4+\sqrt{B}>0$. We have:
	$$
	\begin{aligned}
		B-A_4^2&=-4x^2(x-1)>0  \Longrightarrow (\sqrt{B}-A_4)(\sqrt{B}+A_4)>0\\
		& \Longrightarrow \sqrt{B}-A_4>0, \sqrt{B}+A_4>0 \ \text{or} \  \sqrt{B}-A_4<0, \sqrt{B}+A_4<0.
	\end{aligned}
	$$
	But the second case cannot happen, since then $0<\sqrt{B}<A_4$ implies $B-A_4^2<0$, a contradiction. Therefore, the first case happens and hence $g+h-f>0$.
	\item Finally, we have $f+h-g>0$ by~\eqref{x<1}.
\end{itemize}
\end{proof}
We can carry over the previous result from $\CN$ to the elliptic curve $\mathcal{E}_N$.
\begin{thm}
	\label{th:2}
	A rational triangle with rational quotient $R/r=N>1/4$ exists  if and only if $\mathcal{E}_N$ has a rational non-torsion point
	 $(u,v)$ such that 
	\begin{equation}
		\label{condition}
		0<-\frac{4N(v+2Nu)}{(u-1)(u+4N-1)}<1.
	\end{equation}
	In this case, the sides are given by
	$$f=\frac{s}{2x}(A_1-\sqrt{B}), \ g=sx, \ h=\frac{s}{2x}(A_2+\sqrt{B}),$$
	where $$x=-\frac{4N(v+2Nu)}{(u-1)(u+4N-1)}$$ and $0<s\in\Q$.
\end{thm}

\begin{exa}
	For $N=1$ the rank of $\mathcal{E}_N$ is zero which does not lead to a triangle. For $N=3$, the rank of $\mathcal{E}_N$ is $1$, being generated by $P=(-44,66)$. Since the coordinates of the point $P$ do not satisfy the conditions of Theorem~\ref{th:2}, there is no  triangle corresponding to this point. However, by using the point $[2]P=(3481/16, -226029/64)$, we see that its coordinates satisfy the conditions of Theorem~\ref{th:2} and hence we get (after scaling)
	$$f=98315, \ g=55696, \ h=52371.$$
	We note that there are more triangles for $N=3$. For example, the points 
	\iffalse $$
	\begin{array}{c}
    \dis [4]P=\left(\frac{16322076723481}{363300329536}, -\frac{93684940203017164611}{218977093825846784}\right), \\ [10pt] 
    \dis -P-T_3^+=\left(-\frac{11}{9}, \frac{242}{27}\right),
	\end{array}
    $$ \fi
    \begin{multline*}
    	[4]P=({16322076723481}/{363300329536},\\ -{93684940203017164611}/{218977093825846784})
    	\end{multline*}
    and $$-P-T_3^+=(-{11}/{9}, {242}/{27}),$$
    respectively lead to the triangles
  \iffalse $$\begin{aligned}
    	f&=46822120411340669769, & g&=39352135250471327456, & h&=15634506390670773305, \\
    	f&=25, & g&=27, & h&=8.
    \end{aligned}$$ \fi
    \begin{multline*}
    	(f,g,h)=(46822120411340669769, 39352135250471327456, \\ 15634506390670773305), 
    \end{multline*}
    and
    $$(f,g,h)=(25, 27, 8).$$
    Note that he point $P+T_6^+=(9, -66)$ also gives the latter triangle.
\end{exa}

\begin{exa} \label{exam2}
	Among the integer numbers $1\leq N\leq 50$, the values 
	$$
	\begin{aligned}
		& 3,5,8,10,11,13,15,17,23,24,26,27,29,31,32,33,34,35,36,37,38,39,41,42,43,\\
		& 46,48,50
	\end{aligned}
	$$ 
\iffalse\begin{multline*}
		 3,5,8,10,11,13,15,17,23,24,26,27,29,31,32,33,34,35,36,37,38,39,41,42,43,\\
		 46,48,50
	\end{multline*}\fi
	lead to positive rank elliptic curves $\mathcal{E}_N$. 
\begin{comment}
	\dc{We may add here something like: We apply the following strategy for obtaining our next examples.}
\end{comment}
If we want to find just one triangle for one of these values of $N$, we can proceed as follows.  
Let $P$ be a generator of $\mathcal{E}_N$, and $u(P), v(P)$ its $u$- and $v$-coordinate.
	\begin{itemize}
		\item  If $u(P)v(P)<0$, we use one of the points $P+T_6^+$ or $-P-T_6^+$ depending on $u(P+T_6^+)v(P+T_6^+)$ is negative or positive respectively.
		
		\item If $u(P)v(P)>0$, we use one of the points $-P+T_6^+$ or $P-T_6^+$ depending on $u(-P+T_6^+)v(-P+T_6^+)$ is negative or positive respectively. 
	\end{itemize}
	Using the formulas in Theorem~\ref{th:2} we then get the triangles exhibited in Table~\ref{tab:1}. 
\begin{comment}
\dc{Is the above strategy of finding our example consistent to the end of Section 4 and Section 6 or it isn't important? Here we have used ``$+$" of torsion points while in Section 6 and end of Section 4, we have used ``$-$" of torsion points. I mean we can replace $T_i^+$ with $-T^-_i$ in this section.}
\end{comment}
%	\newpage
\begin{longtable}[c]{cccc}
		\caption{Examples for $1\leq N\leq 50$}\label{tab:1}
	\label{rank1,tor26} 
	\\ \hline 
%	\toprule
	$N$ & $f$ & $g$ & $h$ \\   \hline 
	%\midrule
	&&& \\ [-10pt]
$3$ & $25$ & $27$ & $8$ \\  [2.5pt]
$5$ & $121$ & $147$ & $40$ \\  [2.5pt]
$8$ & $49$ & $50$ & $6$ \\  [2.5pt]
$10$ & $121$ & $128$ & $15$ \\ [2.5pt]
$11$ & $471969$ & $591976$ & $142885$ \\ [2.5pt]
$13$ & $24037$ & $35000$ & $11913$ \\ [2.5pt]
$15$ & $243$ & $245$ & $16$ \\ [2.5pt]
$17$ & $4107$ & $4205$ & $272$ \\ [2.5pt]
$23$ & $363$ & $368$ & $17$ \\ [2.5pt]
$24$ & $242$ & $243$ & $10$ \\ [2.5pt]
$26$ & $1568$ & $1587$ & $65$ \\ [2.5pt]
$27$ & $24900840$ & $26234439$ & $1866059$ \\ [2.5pt]
$29$ & $256$ & $261$ & $11$ \\ [2.5pt]
$31$ & $130355$ & $126736$ & $6171$ \\ [2.5pt]
$32$ & $84568$ & $73947$ &  $11830$ \\ [2.5pt]
$33$ & $1323$ & $1352$ & $55$ \\ [2.5pt]
$34$ & $529$ & $640$ & $119$ \\ [2.5pt]
$35$ & $847$ & $845$ & $24$ \\ [2.5pt]
$36$ & $5043$ & $5415$ & $448$ \\ [2.5pt]
$37$ & $31423$ & $31205$ & $888$ \\ [2.5pt]
$38$ & $150544$ & $164331$ &  $15895$ \\ [2.5pt]
$39$ & $116467264$ & $143721781$ & $28780245$ \\ [2.5pt]
$41$ & $158251147734128961$ & $179454792712801424$ & $23209487182638905$ \\ [2.5pt]
$42$ & $841$ & $864$ & $35$ \\ [2.5pt]
$43$ & $2989137$ & $2805275$ & $219128$ \\ [2.5pt]
$46$ & $18910493440839$ & $18598307793911$ & $571951808000$ \\ [2.5pt]
$48$ & $675$ & $676$ & $14$ \\ [2.5pt]
$50$ & $2401$ & $2535$ & $160$ \\ 
	\bottomrule
\end{longtable}
	 
\end{exa}
\begin{comment}
\dc{Needed any other starting sentence here?}
\end{comment}
In Section~\ref{sec:triangles} we intend to find not only one triangle, but an infinite family of triangles for a given rational value of $N=R/r$ for which the rank of $\mathcal{E}_N$ is positive.
To do so, we first notice that the condition~\eqref{condition} in Theorem~\ref{th:2} can be transformed into a nicer form. 
\begin{thm}
	\label{th:3}
	A rational triangle with rational quotient $R/r=N>1/4$ exists  if and only if $\mathcal{E}_N$ has a rational non-torsion point  $(u,v)$ such that 
	\begin{equation}\label{eq-13'}1-4N < u < 0 \quad \text{or} \quad u > 1.\end{equation}
	In this case, the sides are given by
	$$f=\frac{s}{2x}(A_1-\sqrt{B}), \ g=sx, \ h=\frac{s}{2x}(A_2+\sqrt{B}),$$
	where $$x=-\frac{4N(v+2Nu)}{(u-1)(u+4N-1)}$$ and $0<s\in\Q$.
\end{thm}
\begin{proof}
	We first note that the curve $\mathcal{E}_N$ intersects $u$-axis in the points
	$$u_1=0, \ u_2=-2N^2-2N+1-2N\sqrt{N(N+2)}, \ u_3=-2N^2-2N+1+2N\sqrt{N(N+2)},$$
see Figure~\ref{fig2}.
	According to~\eqref{eq-x}, we observe that if $x=0$, then $v=-2Nu$ and if $x=1$, then 
	$$v=-\frac{u^2+2(4N^2+2N-1)u-(4N-1)}{4N}.$$ 
Hence the region of points $(u,v)$ satisfying condition~\eqref{condition} is the green region
between the two curves
\begin{align*}
c_1:\ v&= -2Nu,\\	
c_2:\ v&=	-\frac{u^2+2(4N^2+2N-1)u-(4N-1)}{4N}.
\end{align*}
The three curves $c_1,c_2$, and $\mathcal{E}_N$ meet in the torsion points
$$
T_6^-=(1 - 4N, -2N (1 - 4N))\text{ and } T_3^-=(1, -2N), \text{ an inflection point of $\mathcal{E}_N$}.
$$
The curves $c_2$ and $\mathcal{E}_N$ actually have a common tangent in $T_6^-$. Moreover, $c_1$ and $\mathcal{E}_N$ meet in $T_2=(0,0)$.
Hence $\mathcal{E}_N$ crosses the boundary of the green region only in $T_6^-,T_3^-$, and in the origin.
It follows that exactly the points $(u,v)$ on $\mathcal{E}_N$ with $1 - 4n<u<1$ or $u>1$ satisfy condition~\eqref{condition}.
	\begin{figure}[htbp]
		\begin{center}
\begin{tikzpicture}[line cap=round,line join=round,x=30,y=22]
\begin{axis}[
    axis line style={draw=none},
    ticks=none,
    xmin=-5.55, xmax=2.8,
    ymin=-4, ymax=4.28]

\addplot [blue,name path = A,
    line width=.8pt,
    domain = -5.55:2.8,
    samples = 10] {-1.4*x} 
    node [very near end, right] {};
    
\addplot [blue,name path = B,
        line width=.8pt,
    domain =  -5.55:2.8,
    samples=100] {-0.3571*x*x - 1.6857*x +0.6429} 
    node [pos=1, above] {};
    
    \addplot [darkgreen!20] fill between [of = A and B, soft clip={domain=-5.55:2.8}];

\addplot [red,
        line width=.8pt,
    domain = -3.3046:0,
    samples=100] {sqrt(x*x*x + 2.76*x*x - 1.8*x)} 
    node [pos=1, above] {};
    
    \addplot [red,
        line width=.8pt,
    domain = -3.3046:0,
    samples=100] {-sqrt(x*x*x + 2.76*x*x - 1.8*x)} 
    node [pos=1, above] {};

\addplot [red,
        line width=.8pt,
    domain = .5447:3,
    samples=100] {sqrt(x*x*x + 2.76*x*x - 1.8*x)} 
    node [pos=1, above] {};
    
    \addplot [red,
        line width=.8pt,
    domain = .5447:3,
    samples=100] {-sqrt(x*x*x + 2.76*x*x - 1.8*x)} 
    node [pos=1, above] {};
    
    \addplot[->, line width=.5pt] coordinates {(-5.55,0) (2.8,0)};
        \addplot[->, line width=.5pt] coordinates {(0,-4) (0,4.28)};
\draw[fill=white] (axis cs:0, 0) circle (1.5pt) node[below,xshift=-8pt,font=\small,yshift=2pt]{$u_1$};
\draw [fill=white] (axis cs:-3.3046817,0) circle (1.5pt) node[anchor=north west,yshift=2pt,xshift=-2pt,font=\small] {$u_2$};
\draw [fill=white] (axis cs:0.5446817918814528,0) circle (1.5pt) node[anchor=north west,yshift=2pt,font=\small] {$u_3$};

\draw [fill=white] (axis cs:-1.8,2.52) circle (1.5pt)  node[font=\small,above,xshift=6.3pt,yshift=-3pt] {$T_6^-$};
\draw [fill=white] (axis cs:1.,-1.4) circle (1.5pt) node[font=\small,right,yshift=1pt,xshift=0pt] {$T_3^-$};
\draw [fill=white] (axis cs:-5.074700720153143,0.) circle (1.5pt)  node {};
\node[red,font=\small] at (axis cs:2.3,3.6){$\mathcal{E}_N$};
\node[blue,font=\small] at (axis cs:-2.3,3.7){$c_1$};
\node[blue,font=\small] at (axis cs:-4.12,1){$c_2$};

\end{axis}
\end{tikzpicture}
		\caption{Only the points on $\mathcal{E}_N$ inside the green region satisfy condition~\eqref{condition}.}
		\label{fig2}
		\end{center}
	\end{figure}
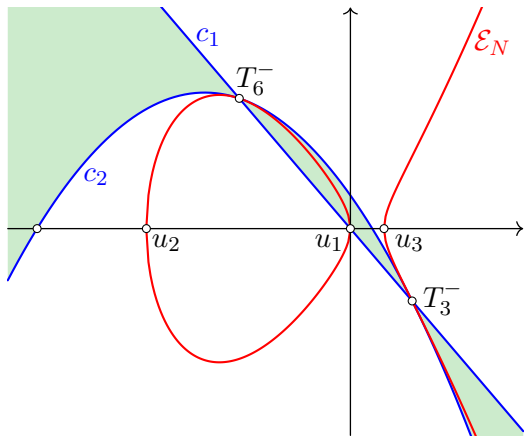
\end{proof}

\section{The cases $N=m^2\pm 1$}\label{sec-3}
The previous section suggests that if $N>1/4$ is a positive rational number for which the rank of $\mathcal{E}_N$ is positive, we may be able to construct integer triangles with $R/r=N$.
So we are facing two problems: Find  rational numbers $N>1/4$ for which the rank of $\mathcal{E}_N$ is positive, and find rational points on such
curves that satisfy condition~\eqref{eq-13'}.
In Section~\ref{sec:triangles} we will show that whenever $\mathcal{E}_N$ has positive rank, the curve has infinitely many rational points satisfying condition~\eqref{eq-13'},
and that these points lead to infinitely many non-similar integer triangles with $R/r=N$. 
In this section we first show that for each $N=m^2\pm 1>1/4$, $m\in\Q$, the rank of $\mathcal{E}_N$ is positive.
We provide explicit formulas for the sides $f,g,h$ of integer triangles with $R/r=N\in\Q$.   

\begin{thm}
    \label{th:4}
   % We have
    \begin{itemize}
    	\item [(i)] If $N=m^2+1$ for a rational number $m>1$ the rank of $\mathcal{E}_N$ is positive with a non-torsion point 
    	$$P=\Bigl(\frac{1}{m^2}, \frac{3m^2+1}{m^3}\Bigr)$$
    	such that the point 
    	\begin{multline*}
    		-P-T_6^+=\Bigl(-\frac{(4m^2+3)(m+1)^2}{(2m^2-m+1)^2}, \\ \frac{2(m+1)(4m^2+3)(m^2+1)(2m^3+m^2+2m-1)}{(2m^2-m+1)^3}\Bigr)
    	\end{multline*}
    	leads to the triangle 
    	$$f=(m-1)(2m^2+m+1)^2, \ g=(m+1)(2m^2-m+1)^2, \ h=4m(m^2+1).$$
    	
    	\item[(ii)] If $N=m^2-1>1/4$  for a rational number $m>1$ the rank of $\mathcal{E}_N$ is positive with a non-torsion point 
    	$$P=\Bigl(\frac{1}{m^2}, \frac{m^2-1}{m^3}\Bigr)$$
    	such that the point
    	$$-P-T_6^+=\Bigl(-\frac{4m^2-5}{(2m-1)^2}, \frac{2(m-1)(4m^2-5)(2m^2+m+1)}{(2m-1)^3}\Bigr)$$
    	 leads to the triangle 
    	$$f=(m-1)(2m+1)^2, \ g=(m+1)(2m-1)^2, \ h=4m.$$
    \end{itemize}    
    %Moreover, a geometric description shows that for each case we have infinitely many triangles.  
\end{thm}
\begin{proof}
(i) By enforcing $u=N-1$ on $\mathcal{E}_N$, the expression $N-1$ must be a square, say $m^2$ for some rational number $m$. 
Then, $N=m^2+1$ and the corresponding elliptic curve $\mathcal{E}_N$ owns the rational non-torsion  point
$$P=\Bigl(\frac{1}{m^2}, \frac{3m^2+1}{m^3}\Bigr).$$
By the addition law, we compute 
\begin{multline*}
	-P-T_6^+=\Bigl(-\frac{(4m^2+3)(m+1)^2}{(2m^2-m+1)^2}, \\ \frac{2(m+1)(4m^2+3)(m^2+1)(2m^3+m^2+2m-1)}{(2m^2-m+1)^3}\Bigr).
\end{multline*}
By taking 
%$$u=-\frac{(4m^2+3)(m+1)^2}{(2m^2-m+1)^2}, \  v=\frac{2(m+1)(4m^2+3)(m^2+1)(2m^3+m^2+2m-1)}{(2m^2-m+1)^3},$$
%\dc{instead we can write ``By taking $(u,v)=-P-T_6^+$".}
$(u,v)=-P-T_6^+$
we observe that condition~\eqref{eq-13'} holds, hence by Theorem~\ref{th:1} we get 
$$f=(m-1)(2m^2+m+1)^2, \ g=(m+1)(2m^2-m+1)^2, \ h=4m(m^2+1).$$
(ii) In a similar fashion, by enforcing $u=N+1$ on $\mathcal{E}_N$, the expression $N+1$ must be a square, say $m^2$ for some rational number $m$. 
Then, $N=m^2-1$ and the corresponding elliptic curve $\mathcal{E}_N$ owns the rational non-torsion  point
$$P=\Bigl(\frac{1}{m^2}, \frac{m^2-1}{m^3}\Bigr).$$
By the addition law, we compute 
$$-P-T_6^+=\Bigl(-\frac{4m^2-5}{(2m-1)^2}, \frac{2(m-1)(4m^2-5)(2m^2+m+1)}{(2m-1)^3}\Bigr).$$
By taking 
%$$u=-\frac{4m^2-5}{(2m-1)^2}, \quad v= \frac{2(m-1)(4m^2-5)(2m^2+m+1)}{(2m-1)^3},$$
$(u,v)=-P-T_6^+$
we observe that  condition~\eqref{eq-13'} holds, hence by Theorem~\ref{th:1} we get 
$$f=(m-1)(2m+1)^2, \ g=(m+1)(2m-1)^2, \ h=4m.$$
\end{proof}

\iffalse
\begin{rmk}
Consider $N=m^2+n\pm 1$. Then, imposing the point $u=\frac{1}{N\mp 1}$ on \eqref{9} implies to have 
\begin{equation}
	\label{Pell}
	m^2+n=k^2
\end{equation}
for some rational $k$. The latter quadratic is a Pell type equation. Write $n=n_1\cdot n_2$ such that the equations 
$$m_1^2+n_1=k_1^2, \quad m_2^2+n_2=k_2^2$$
has respectively the solutions $(m_1,k_1)$ and $(m_2,k_2)$ with $m_1m_2\neq 0$. Then, \eqref{Pell} has the solutions $$(k_1m_2\pm k_2m_1, k_1k_2\pm m_1m_2).$$ 
By composing the solutions, infinitely many integer solutions are obtained. But, substituting \eqref{Pell} in $N$ leads to the case discussed in Section~\ref{sec-3}. Note that when $n$ is a square, we are also led to the cases discussed in the previous section. 
\end{rmk}
\fi

\section{Infinitely many triangles for curves of positive rank}\label{sec:triangles}
We start by determining just one rational point on $\mathcal{E}_N$ that fulfills condition~(\ref{eq-13'}).
This point will then be the stepping stone for an infinite number of such points.
\begin{comment}
\dc{Needed to start with something here? For example emphasizing the main result of the section or paper? Also, in the proofs of the lemmas in this section, the word
``condition" has been used several times and we can remove some (?)}
\end{comment}
\begin{lem}\label{lem:1a}
If $\mathcal{E}_N$ has positive rank for a rational number $N>1/4$, then there are rational points on $\mathcal{E}_N$ that satisfy condition~\eqref{eq-13'}.
\end{lem}
\begin{proof}
Let $R=(u_R,v_R)$ be a rational non-torsion point on $\mathcal{E}_N$ that does not satisfy condition~\eqref{eq-13'}.
If $u_R<1-4n$, then $R+T_3^-$ is a rational non-torison point satisfying condition~\eqref{eq-13'}.
If $0<u_R<1$, then $-(R+T_6^-)$ is a rational non-torison point satisfying condition~\eqref{eq-13'}.
See Figure~\ref{fig:3}.
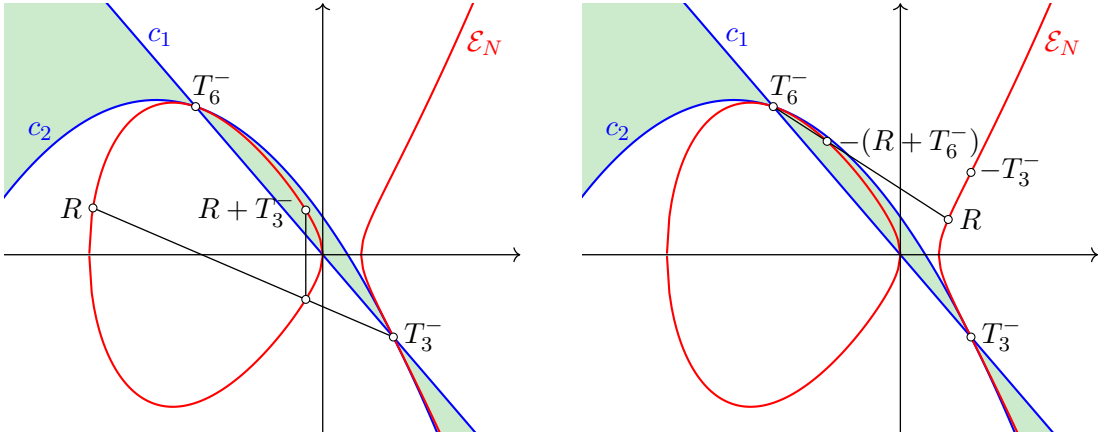
\begin{figure}[h!]
\begin{center}
\begin{tikzpicture}[line cap=round,line join=round,x=30,y=22]
\begin{axis}[
    axis line style={draw=none},
    ticks=none,
    xmin=-4.5, xmax=2.8,
    ymin=-3, ymax=4.28]

\addplot [blue,name path = A,
    line width=.8pt,
    domain = -5.55:2.8,
    samples = 10] {-1.4*x} 
    node [very near end, right] {};
    
\addplot [blue,name path = B,
        line width=.8pt,
    domain =  -5.55:2.8,
    samples=100] {-0.3571*x*x - 1.6857*x +0.6429} 
    node [pos=1, above] {};
    
    \addplot [darkgreen!20] fill between [of = A and B, soft clip={domain=-5.55:2.8}];

\addplot [red,
        line width=.8pt,
    domain = -3.3046:0,
    samples=100] {sqrt(x*x*x + 2.76*x*x - 1.8*x)} 
    node [pos=1, above] {};
    
    \addplot [red,
        line width=.8pt,
    domain = -3.3046:0,
    samples=100] {-sqrt(x*x*x + 2.76*x*x - 1.8*x)} 
    node [pos=1, above] {};

\addplot [red,
        line width=.8pt,
    domain = .5447:3,
    samples=100] {sqrt(x*x*x + 2.76*x*x - 1.8*x)} 
    node [pos=1, above] {};
    
    \addplot [red,
        line width=.8pt,
    domain = .5447:3,
    samples=100] {-sqrt(x*x*x + 2.76*x*x - 1.8*x)} 
    node [pos=1, above] {};
    
    \addplot[->, line width=.5pt] coordinates {(-5.55,0) (2.8,0)};
    \addplot[->, line width=.5pt] coordinates {(0,-4) (0,4.28)};
%\draw[fill=white] (axis cs:0, 0) circle (1.5pt) node[below,xshift=-8pt,font=\small,yshift=2pt]{$u_1$};
%\draw [fill=white] (axis cs:-3.3046817,0) circle (1.5pt) node[anchor=north west,yshift=2pt,xshift=-2pt,font=\small] {$u_2$};
%\draw [fill=white] (axis cs:0.5446817918814528,0) circle (1.5pt) node[anchor=north west,yshift=2pt,font=\small] {$u_3$};

\addplot [line width=0.5pt] coordinates {(-3.2534338076472844,0.7957808291336159)  (1.,-1.4)};
\addplot [line width=0.5pt] coordinates {(-0.24006532909130973,-0.759832126317231) (-0.24006532909130973,0.7598321263172303)};

\draw [fill=white] (axis cs:-1.8,2.52) circle (1.5pt)  node[font=\small,above,xshift=6.3pt,yshift=-3pt] {$T_6^-$};
\draw [fill=white] (axis cs:1.,-1.4) circle (1.5pt) node[font=\small,right,yshift=1pt,xshift=0pt] {$T_3^-$};

\draw [fill=white] (axis cs:-3.2534338076472844,0.7957808291336159) circle (1.5pt) node[font=\small,left,xshift=0pt,yshift=0pt] {$R$};
\draw [fill=white] (axis cs:-0.24006532909130973,-0.759832126317231) circle (1.5pt);
\draw [fill=white] (axis cs:-0.24006532909130973,0.7598321263172303) circle (1.5pt)  node[font=\small,left,xshift=0pt,yshift=0pt] {$R+T_3^-$};

\draw [fill=white] (axis cs:-5.074700720153143,0.) circle (1.5pt)  node {};
\node[red,font=\small] at (axis cs:2.3,3.6){$\mathcal{E}_N$};
\node[blue,font=\small] at (axis cs:-2.3,3.7){$c_1$};
\node[blue,font=\small] at (axis cs:-4,2.1){$c_2$};

\end{axis}
\end{tikzpicture}\qquad
\begin{tikzpicture}[line cap=round,line join=round,x=30,y=22]
\begin{axis}[
    axis line style={draw=none},
    ticks=none,
    xmin=-4.5, xmax=2.8,
    ymin=-3, ymax=4.28]

\addplot [blue,name path = A,
    line width=.8pt,
    domain = -5.55:2.8,
    samples = 10] {-1.4*x} 
    node [very near end, right] {};
    
\addplot [blue,name path = B,
        line width=.8pt,
    domain =  -5.55:2.8,
    samples=100] {-0.3571*x*x - 1.6857*x +0.6429} 
    node [pos=1, above] {};
    
    \addplot [darkgreen!20] fill between [of = A and B, soft clip={domain=-5.55:2.8}];

\addplot [red,
        line width=.8pt,
    domain = -3.3046:0,
    samples=100] {sqrt(x*x*x + 2.76*x*x - 1.8*x)} 
    node [pos=1, above] {};
    
    \addplot [red,
        line width=.8pt,
    domain = -3.3046:0,
    samples=100] {-sqrt(x*x*x + 2.76*x*x - 1.8*x)} 
    node [pos=1, above] {};

\addplot [red,
        line width=.8pt,
    domain = .5447:3,
    samples=100] {sqrt(x*x*x + 2.76*x*x - 1.8*x)} 
    node [pos=1, above] {};
    
    \addplot [red,
        line width=.8pt,
    domain = .5447:3,
    samples=100] {-sqrt(x*x*x + 2.76*x*x - 1.8*x)} 
    node [pos=1, above] {};
    
    \addplot[->, line width=.5pt] coordinates {(-5.55,0) (2.8,0)};
    \addplot[->, line width=.5pt] coordinates {(0,-4) (0,4.28)};
%\draw[fill=white] (axis cs:0, 0) circle (1.5pt) node[below,xshift=-8pt,font=\small,yshift=2pt]{$u_1$};
%\draw [fill=white] (axis cs:-3.3046817,0) circle (1.5pt) node[anchor=north west,yshift=2pt,xshift=-2pt,font=\small] {$u_2$};
%\draw [fill=white] (axis cs:0.5446817918814528,0) circle (1.5pt) node[anchor=north west,yshift=2pt,font=\small] {$u_3$};

\addplot [line width=0.5pt] coordinates {(0.6778561097171775,0.5995963225719247) (-1.8,2.52)};

\draw [fill=white] (axis cs:-1.8,2.52) circle (1.5pt)  node[font=\small,above,xshift=6.3pt,yshift=-3pt] {$T_6^-$};
\draw [fill=white] (axis cs:1.,-1.4) circle (1.5pt) node[font=\small,right,yshift=1pt,xshift=0pt] {$T_3^-$};
\draw [fill=white] (axis cs:1.,1.4) circle (1.5pt) node[font=\small,right,yshift=1pt,xshift=-1pt] {$-T_3^-$};

\draw [fill=black]  circle (1.5pt);
\draw [fill=white] (axis cs:0.6778561097171775,0.5995963225719247) circle (1.5pt) node[font=\small,right,xshift=0pt,yshift=0pt] {$R$};
%\draw [fill=white] (axis cs:-0.24006532909130973,-0.759832126317231) circle (1.5pt);
\draw [fill=white] (axis cs:-1.037190330047419,1.9288024362634044) circle (1.5pt)  node[font=\small,right,xshift=0pt,yshift=0pt] {$-(R+T_6^-)$};

\draw [fill=white] (axis cs:-5.074700720153143,0.) circle (1.5pt)  node {};
\node[red,font=\small] at (axis cs:2.3,3.6){$\mathcal{E}_N$};
\node[blue,font=\small] at (axis cs:-2.3,3.7){$c_1$};
\node[blue,font=\small] at (axis cs:-4,2.1){$c_2$};

\end{axis}
\end{tikzpicture}

\caption{Rational points on $\mathcal{E}_N$, satisfying condition~\eqref{eq-13'}.}\label{fig:3}
\end{center}
\end{figure}
\end{proof}
\begin{lem}\label{lem:4}
If $\mathcal{E}_N$ has positive rank for a rational number $N>1/4$, then there are rational points $R=(u_R,v_R)$ on $\mathcal{E}_N$ that satisfy condition~\eqref{eq-13'} and $u_R>1$.
\end{lem}
\begin{proof}
According to Lemma~\ref{lem:1a}, there are rational non-torsion points that satisfy condition~\eqref{eq-13'}.
Now, suppose $1-4N<u_R<0$, then $R+T_6^--T_3^-$ satisfies condition~\eqref{eq-13'} and the $u$-coordinate of this point is larger than $1$.
\end{proof}
\begin{lem}\label{lll}
If $\mathcal{E}_N$ has positive rank for a rational number $N>1/4$, then there are infinitely  many rational points $R=(u_R,v_R)$ on $\mathcal{E}_N$ that satisfy $u_R>1$.
\end{lem}
\begin{proof}
According to Lemma~\ref{lem:4}, there are rational non-torsion points that satisfy condition~\eqref{eq-13'} and
whose $u$-coordinates are larger than $1$. If $R_0=(u_0,v_0)$ is such a point, then the point $-(2R_0+T_3^-)$ is also a
point that satisfy condition~\eqref{eq-13'} and whose $u$-coordinate is  larger than $1$. Iterating this process yields a sequence
$R_0, R_k=(-1)^k (2^k R_0+J_k T_3^-)$, where $(J_k)_k$ is the Jacobsthal  sequence modulo 3, i.e., $J_k\equiv -k \mod  3$. 
\begin{comment}
\dc{Why? Reference and a bit more explanation!} 
\end{comment}
Recall that 
$T_3^-$ is a torsion point of order~$3$. In particular, the sequence $R_k$ is not periodic, and infinitely many non-similar
triangles result from the sequence of points $R_k$.
\end{proof}
As   a result, we obtain the following.
\begin{cor}
If $\mathcal{E}_N$ has positive rank for a rational number $N>1/4$, then there are infinitely many rational and integer triangles with $R/r=N$.
\end{cor}

\section{Poncelet's theorem}\label{sec:poncelet}
Our results now allow us to make an interesting statement with regard to Poncelet's closure theorem. 
\begin{comment}\dc{Reference?} 
\cb{Poncelet's Theorem is very well known.}
\end{comment}
\begin{cor}
Let $C$ and $E$ 
\begin{comment}
\dc{As suggested earlier, we may use the notations $C_r$ and $C_R$}
\cb{we keep $E$ and $C$.}
\end{comment} 
be two circles such that the ratio of their radii is rational and such that a triangle
with rational sides has  $C$ as circumcircle and  $E$ as one of its excircles. Then there exist infinitely
many rational non-similar triangles with rational sides that have  $C$ as circumcircle and  $E$ as excircle.
\end{cor}
Figure~\ref{fig:Poncelet} shows three such Poncelet %\dc{or Poncelet's} 
triangles with $R=5/2, r=2$:
\iffalse  $$\begin{aligned}
	f_1&=5& g_1&=4 & h_1&=3  & u&=-\frac14 &v&=\frac54\\
	f_2&=\frac{204}{65}&g_2&=\frac{416}{85}&h_2&=\frac{700}{221}& u&=\frac{169}{64} &v&=-\frac{4355}{512}\\
	f_3&=\frac{6757}{5513}&g_3&=\frac{37697}{8621}&h_3&=\frac{126540}{34717}& u&=-\frac{64009}{22201} &v&=\frac{53116085}{6615898}
\end{aligned}$$  \fi
\begin{alignat*}{5}
f_1&=5\qquad\qquad& g_1&=4\qquad\qquad& h_1&=3   \qquad\quad\quad& u&=-\frac14 \quad\quad\quad&v&=\frac54\\%R
f_2&=\frac{204}{65}&g_2&=\frac{416}{85}&h_2&=\frac{700}{221}& u&=\frac{169}{64} &v&=-\frac{4355}{512}\\%-2R
f_3&=\frac{6757}{5513}&g_3&=\frac{37697}{8621}&h_3&=\frac{126540}{34717}& u&=-\frac{64009}{22201} &v&=\frac{53116085}{6615898}
\end{alignat*}

\begin{figure}[h!]
\begin{center}
\definecolor{darkgreen}{rgb}{0.,.6,0.}
\begin{tikzpicture}[line cap=round,line join=round,x=33,y=33]
\clip(-0.645015958872033,-1.126532812635162) rectangle (8.078104646295522,4.125051470475843);

\draw [line width=.8pt] (2.,1.5) circle (2.5);
\draw [line width=.8pt] (6.,2.) circle (2.);
\draw [line width=0.4pt,domain=0.0:8.078104646295522] plot(\x,{(-0.-0.*\x)/4.});
\draw [line width=0.4pt,domain=0.0:8.078104646295522] plot(\x,{(-0.--3.*\x)/4.});
\draw [line width=0.4pt,domain=0.6560000000000001:8.078104646295522] plot(\x,{(-2.1244970414201183--0.3476449704142018*\x)/3.119147928994082});
\draw [line width=0.4pt,domain=0.6560000000000001:8.078104646295522] plot(\x,{(-4.375916955017303--3.5007335640138417*\x)/3.420124567474048});
\draw [line width=0.4pt,domain=1.9550261279991028:8.078104646295522] plot(\x,{(-1.7384715928463053--0.27901805526407353*\x)/1.1934668366282712});
\draw [line width=0.4pt,domain=1.9550261279991028:8.078104646295522] plot(\x,{(-9.58853541833022--3.78532480320648*\x)/2.1890121174593196});
\draw[line width=1.pt,color=blue,fill=blue,fill opacity=0.10000000149011612] (0.656,-0.608) -- (3.775147928994083,-0.26035502958579837) -- (4.076124567474048,2.89273356401384) -- cycle;
\draw[line width=1.pt,color=red,fill=red,fill opacity=0.10000000149011612] (0.,0.) -- (4.,0.) -- (4.,3.) -- cycle;
\draw[line width=1.pt,color=darkgreen,fill=darkgreen,fill opacity=0.10000000149011612] (1.9550261279991028,-0.9995954374332752) -- (3.148492964627374,-0.7205773821692016) -- (4.144038245458423,2.785729365773205) -- cycle;

\begin{small}
\draw [fill=white] (0.,0.) circle (1.5pt);
\draw [fill=white] (4.,0.) circle (1.5pt);
\draw [fill=white] (4.,3.) circle (1.5pt);
\draw[color=black] (0.8,3.5) node {$C$};
\draw[color=black] (7.66,3.5) node {$E$};
\draw [fill=white] (0.656,-0.608) circle (1.5pt);
\draw [fill=white] (3.775147928994082,-0.2603550295857983) circle (1.5pt);
\draw [fill=white] (4.076124567474048,2.8927335640138416) circle (1.5pt);
\draw [fill=white] (0.,0.) circle (1.5pt);
\draw [fill=white] (1.9550261279991028,-0.9995954374332752) circle (1.5pt);
\draw [fill=white] (3.148492964627374,-0.7205773821692016) circle (1.5pt);
\draw [fill=white] (4.144038245458423,2.785729365773205) circle (1.5pt);
\end{small}
\end{tikzpicture}
\caption{Poncelet's theorem for triangles.}\label{fig:Poncelet}
\end{center}
\end{figure}
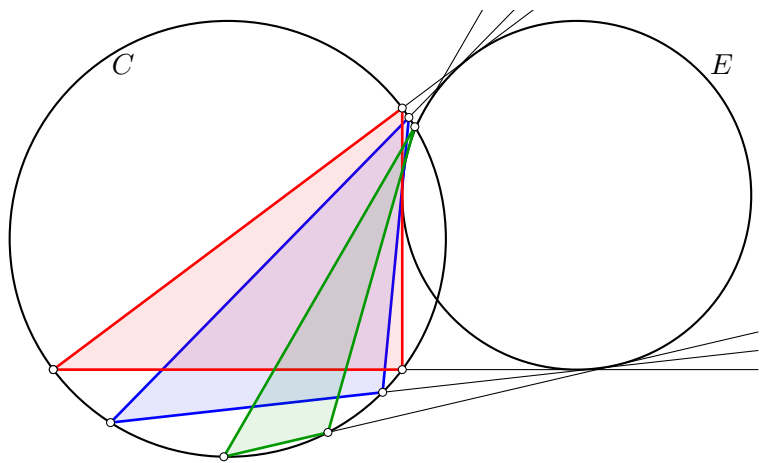

%\dc{What about enclosing the paper with some question(s)?}

%\section{Conclusion and open problems}\label{sec-4}
%In this work, we studied the problem of circumcircle and exircle with integer $R/r$ and we were lucky to find two closed formulas for $N$ for which we constructed infinitely many such triangles. It is challenging to find new closed formulas if any. 
%\subsubsection*{Acknowledgement} 
%We are grateful to the referees for the valuable comments that  helped improve this article. The third author would like to thank ETH Z\"urich for its hospitality.

%\noindent{\bf Declarations of interest:} none
%%%%%%%%%%%%%%%%%%%%%%%%%%%%%%%%%%%%%%%%%%%%%%%%%%%%%%%%%%

%\section*{Acknowledgment}
%We would like to thank the referee for his or her careful 
%reading and useful %comments and suggestions, which helped to 
%improve the quality of the article.

%%%%%%%%%%%%%%%%%%%%%%%%%%%%%%%%%%%%%%%%%%%%%%%%%%%%%%%%%%

\bibliographystyle{plain}
%\bibliography{kubert}

\end{document}